\documentclass[11pt]{amsart}
\usepackage{hyperref,amssymb,tikz}
\usepackage[margin=2.5cm]{geometry}
\usetikzlibrary{arrows,shapes,positioning,decorations.markings, decorations.pathreplacing,decorations.pathmorphing}
\newcommand{\Cay}{\mathop{\mathrm{Cay}}}
\newcommand{\Aut}{\mathop{\mathrm{Aut}}}

\def\cent#1#2{{\bf C}_{{#1}}({{#2}})}

\def\nor#1#2{{\bf N}_{{#1}}({{#2}})}
\numberwithin{equation}{section}
\newtheorem{theorem}{Theorem}[section]
\newtheorem{lemma}[theorem]{Lemma}
\newtheorem{corollary}[theorem]{Corollary}

\newtheorem{example}[theorem]{Example}
\newtheorem{conjecture}[theorem]{Conjecture}
\newtheorem{question}[theorem]{Problem}
\theoremstyle{definition}
\newtheorem{definition}[theorem]{Definition}

\begin{document}
\title[Enumeration bipartite Cayley graphs]{On the existence and the enumeration of bipartite regular representations of Cayley graphs over abelian groups}

\author{Jia-Li Du}
\address{Jia-Li Du, Department of Mathematics, Beijing Jiaotong University, Beijing 100044, China}
\email{JiaLiDu@bjtu.edu.cn}

\author{Yan-Quan Feng}
\address{Yan-Quan Feng, Department of Mathematics, Beijing Jiaotong University, Beijing 100044, China}
\email{yqfeng@bjtu.edu.cn}

\author{Pablo Spiga}
\address{Pablo Spiga, Dipartimento di Matematica e Applicazioni, University of Milano-Bicocca, Via Cozzi 55, 20125 Milano, Italy}
\email{pablo.spiga@unimib.it}

\begin{abstract}
In this paper we are interested in the asymptotic enumeration of bipartite Cayley digraphs
 and Cayley graphs over abelian groups. Let $A$ be an abelian group and let $\iota$ be the automorphism of $A$
defined by $a^\iota=a^{-1}$, for every $a\in A$. A Cayley graph $\Cay(A, S)$ is said to have an automorphism group as small as possible if $\Aut(\Cay(A,S)) = \langle A,\iota\rangle$. In this paper,
we show that, except for two infinite families, almost all bipartite Cayley graphs on abelian groups have automorphism group as small as possible. We also  investigate the analogous question for bipartite Cayley digraphs.

These results are used for the asymptotic enumeration of bipartite Cayley digraphs and graphs over abelian groups.
\smallskip

\end{abstract}

\maketitle

\section{Introduction}\label{s: intro}
All digraphs and groups considered in this paper are finite.
Let $G$ be a group and let $S$ be a subset of $G$. The \textit{Cayley digraph} on $G$ with
connection set $S$, denoted $\Cay(G, S)$, is the digraph with vertex-set $G$ and with
$(g, h)$ being an arc if and only if $gh^{-1} \in S$.  It is easy to see that $\Cay(G, S)$ is a graph if and only if $S$ is inverse-closed (that is, $S^{-1}:=\{s^{-1}\mid s\in S\}=S$), in which
case it is called a \textit{Cayley graph}. It is also easy to check that $G$ acts regularly as a
group of automorphisms on $\Cay(G, S)$ by right multiplication. Therefore, in what follows, we always identify $G$ as a subgroup of the automorphism group $\Aut(\Cay(G,S))$ of $\Cay(G,S)$.

In the extreme case, when $G$ equals $\Aut(\Cay(G,S))$, $\Cay(G, S)$
is called a \textit{DRR} (for \textit{digraphical regular representation}). A DRR which is a graph
is called a \textit{GRR} (for \textit{graphical regular representation}).
DRRs and GRRs have been widely studied. There are two natural questions on DRRs (and on GRRs):
\begin{itemize}
\item which groups admit a DRR (or a GRR)?
\item when the size of $G$ tends to infinity, what is the probability that a Cayley digraph (respectively, graph) over $G$ is a DRR (respectively, GRR)?
\end{itemize}
We do not intend to give here a full account on the study of GRRs, this  involved many researchers and papers. Some of the most influential work along the way is due to Babai, Godsil, Hetzel, Imrich, Notwitz and Watkins (to name a few), see \cite{IW,NW,NW1,W1,W2,W3}.

The answer to the first question for DRRs was given by Babai~\cite{babai1}. The analogous answer for GRRs turns out to be considerably harder and  was  completed by Godsil~\cite{Godsil}, after a long series of partial results by various authors, see~\cite{Hetzel,Imrich2,NW} for example. Once this  has been established for DRRs and GRRs, it continued to be considered for a large variety of natural Cayley (di)graphs: for instance, oriented regular representations~\cite{morrisspiga,morrisspiga1,spiga}, tournament regular representations~\cite{babai2}, graphical Frobenius representations~\cite{DWT,WT,spiga0,spiga1}, and graphical representations of small valency~\cite{Spiga,xf,XF}.

The second question seems dramatically harder and it has been touched only in a few particular cases and in peculiar situations, see~\cite{DSV,MSV}. (There are some recent results on the asymptotic enumeration of DRRs in~\cite{MS}.) The asymptotic enumeration of vertex-transitive graphs seems also rather difficult, and we refer the interested reader to the seminal work of Mckay and Praeger~\cite{MP}  and to~\cite{PSV} for some more recent results on vertex-transitive graphs of fixed valency.

The general aim of  this paper is to understand and construct  bipartite DRRs and  bipartite GRRs. The standard techniques developed in~\cite{babai1,Godsil,Hetzel} involving a local analysis on the neighborhood of a Cayley (di)graph do not seem to work for bipartite graphs, because the neighborhood of a bipartite graph is the empty graph, which brings little or no information. Therefore, in our paper, we start our investigation by considering bipartite Cayley digraphs and graphs over abelian groups: this allows us to apply the group-theoretic techniques in~\cite{DSV}.

\begin{theorem}\label{thrm:1}Let $A$ be an abelian group  and let $B$ be a subgroup of $A$ having index $2$. The number of subsets $S$ of $A\setminus  B$ such that $\Cay(A,S)$  is a bipartite $\mathrm{DRR}$  is at least $2^{\frac{|A|}{2}}-3\cdot 2^{\frac{3|A|}{8}+(\log_2|A|)^2}$.
\end{theorem}

Since $A\setminus B$ has $2^{|A\setminus B|}=2^{\frac{|A|}{2}}$ subsets, from Theorem \ref{thrm:1} we immediately obtain the following corollary.

\begin{corollary}\label{cor:1}For every positive real number $\varepsilon>0$, there exists a natural number $n_\varepsilon$ such that, for every  abelian group $A$ of order at least $n_\varepsilon$ and for every subgroup $B$ of $A$ having index $2$, we have
$$\frac{|\{S\mid S\subseteq A\setminus B, \Cay(A,S) \,\mathrm{is\,a\, DRR}\}|}{|\{S\mid S\subseteq A\setminus B\}|}\ge 1-\varepsilon.$$
\end{corollary}
Broadly speaking Corollary \ref{cor:1} says that, when the order of an abelian group $A$ is even and sufficiently large, most bipartite Cayley digraphs over $A$ are DRRs. It worth stressing that Corollary \ref{cor:1} says something slightly stronger, that is, if a subgroup $B$ of $A$ of index $2$ is given in advance, most Cayley digraphs over $A$ with bipartition $\{B,A\setminus B\}$ are DRRs. The difference seems rather subtle, but it is remarkably important for undirected graphs as we will discuss in detail later, see Theorems \ref{thrm:aa} and \ref{thrm:11}.

\begin{corollary}\label{cor:1new}Let $A$ be an abelian group of even order and let $\mathcal{S}:=\{S\subseteq A\mid \Cay(A,S)\, \mathrm{ bipartite}\}$. The proportion of subsets $S$ of $\mathcal{S}$ such that $\Cay(A,S)$ is a bipartite $\mathrm{DRR}$ tends to $1$ as $|A|\to \infty$.
\end{corollary}

Given a positive integer $t$, we denote by $C_t$ the cyclic group of order $t$.
Since the estimate in Theorem \ref{thrm:1} is rather explicit, we also obtain the following corollary.

\begin{corollary}\label{cor:2}Let $A$ be an abelian group  and let $B$ be a subgroup of $A$ having index $2$. Then, either there exists a subset $S$ of $A\setminus B$ such that $\Cay(A,S)$ is a bipartite  $\mathrm{DRR}$  or the pair $(A,B)$ is in Table~$\ref{table:1}$.
\end{corollary}
\begin{table}[!ht]
\begin{tabular}{|c|c|c|}\hline
$A$&$B$&directed bipartite Cayley index\\\hline
$C_2\times C_2$&$C_2$&$2$\\
$C_2\times C_2\times C_2$&$C_2\times C_2$&$6$\\
$C_2\times C_2\times C_2\times C_2$&$C_2\times C_2\times C_2$&$24$\\
$C_2\times C_2\times C_2\times C_2\times C_2$&$C_2\times C_2\times C_2\times C_2$&$72$\\
$C_2\times C_2\times C_2\times C_2\times C_2\times C_2$&$C_2\times C_2\times C_2\times C_2\times C_2$&$4$\\

$C_3\times C_6$&$C_3\times C_3$&$2$\\
$C_4\times C_2\times C_2\times C_2$&$C_2\times C_2\times C_2\times C_2$&$4$\\

$C_4\times C_2\times C_2$&$C_2\times C_2\times C_2$&4\\

$C_4\times C_2\times C_2$&$C_4\times C_2$&2\\

$C_4\times C_2$&$C_2\times C_2$&$2$\\\hline
\end{tabular}
\label{table:1}
\caption{Abelian groups and their index $2$ subgroups not admitting a bipartite DRR}
\end{table}
In line with the work of Morris and Tymburski~\cite{MT} and with the pioneering work of Imrich and Watkins \cite{IW}, we have included a third column in Table~\ref{table:1}, whose meaning we now explain. Let $G$ be a group, let $S$ be a subset of $G$ and let $\Gamma:=\Cay(G,S)$ be the Cayley digraph over $G$ with connection set $S$. The \textit{Cayley index} $c(\Gamma)$ of the digraph $\Gamma$ is $|\Aut(\Gamma):G|$. Therefore $c(\Gamma)$ measures the degree of symmetry of a Cayley digraph; intuitively, the larger $c(\Gamma)$ is, the more symmetric $\Gamma$ is. Moreover, $c(\Gamma)$ is somehow unbiased with respect to the number of vertices of $\Gamma$. (We observe here that Morris and Tymburski define and consider the Cayley index only for undirected graphs, see~\cite[Definition~$1.1$]{MT}.) Following the line of research of Imrich and Watkins, we give the following definition.

\begin{definition}\label{def:1}
{Let $G$ be a group and let $B$ be a subgroup of $G$ having index $2$. The \textit{directed bipartite Cayley index }$\vec{c}(G,B)$ of $(G,B)$ is $$\vec{c}(G,B):=\min_{S\subseteq G\setminus B}|\Aut(\Cay(G,S)):G|.$$
We also define, for groups admitting an index $2$ subgroup, the \textit{global directed bipartite Cayley index}
$$\vec{c}_b(G):=\min_{B\le G, |G:B|=2}\vec{c}(G,B).$$
}
\end{definition}
In the light of Definition~\ref{def:1}, Corollary~\ref{cor:2} says that, except for the ten exceptions in Table~\ref{table:1}, $\vec{c}(A,B)=1$ for every abelian group $A$  and for every subgroup $B$ of $A$ having index $2$. In the third column of Table~\ref{table:1}, we determine the directed bipartite Cayley index for the ten exceptional pairs.

We also prove the following unlabeled version of Theorem~\ref{thrm:1}.
\begin{theorem}\label{thrm:unlabelled1}Let $A$ be an abelian group  and let $B$ be a subgroup of $A$ having index $2$. Then, the number of bipartite Cayley digraphs (up to graph-isomorphism) over $A$ with bipartition $\{B,A\setminus B\}$ is at least $2^{\frac{|A|}{2}-(\log_2|A|)^2}-3\cdot 2^{\frac{3|A|}{8}}$. Moreover, among all bipartite Cayley digraphs (up to graph-isomorphism) over $A$ with bipartition $\{B,A\setminus B\}$, the proportion that are $\mathrm{DRR}$s tends to $1$ as $|A|\to\infty$.
\end{theorem}

In this paper we also consider bipartite Cayley graphs
over abelian groups $A$. We denote by $\iota:A\to A$ the automorphism of the abelian group $A$ mapping each element to its inverse, that is, $a^\iota=a^{-1}$ for every $a\in A$. Clearly, $\iota$ is the identity mapping when $A$ has exponent $2$, and $\iota$ is an involutory automorphism when $A$ has exponent greater than $2$. When $A$ has exponent greater than $2$ no Cayley graph is a GRR,
because $\iota$ is a non-identity  graph automorphism. Therefore, we are interested in bipartite Cayley graphs $\Cay(A,S)$ having automorphism group ``as small as possible'', that is, $\Aut(\Cay(A,S))=\langle A,\iota\rangle$. We formalize this idea in the following definition.

\begin{definition}\label{def:2}
{Let $G$ be a group and let $B$ be a subgroup of $G$ having index $2$. The \textit{bipartite Cayley index }$c(G,B)$ of $(G,B)$ is $$c(G,B):=\min_{\substack{S\subseteq G\setminus B\\S=S^{-1}}}|\Aut(\Cay(G,S)):G|.$$
We define, for groups admitting an index $2$ subgroup, the \textit{global bipartite Cayley index}
$$c_b(G):=\min_{B\le G, |G:B|=2}c(G,B).$$}
\end{definition}
When $A$ is abelian of exponent greater than $2$, $c(A,B)\ge 2$ for every subgroup $B$ of $A$ having index $2$. Moreover, $c(G,B)\ge c(G)$, where $c(G)$ is the Cayley index of $G$ as defined in~\cite[Definition~$1.1$]{MT}.

\begin{theorem}\label{thrm:aa}
If $A\cong C_4\times C_2^\ell$ for some $\ell\ge 1$ and  $B\cong C_2^{\ell+1}$, or  $A\cong C_4^2\times C_2^\ell$ for some $\ell\ge 0$ and $B\cong C_4\times C_2^{\ell+1}$,
then there exists no subset $S\subseteq A\setminus B$ such that $\Cay(A,S)$ is a bipartite Cayley graph  with bipartite Cayley index $2$.
\end{theorem}
In particular Theorem \ref{thrm:aa} shows that there exist two infinite families  with $c(A,B)>2$.  This behavior is a novelty compared with the statement of Theorem~\ref{thrm:1}.

\begin{question}{\rm Determine the bipartite Cayley index for the two exception families in Theorem~\ref{thrm:aa}. That is, determine $c(A,B)$, where
$A\cong C_4\times C_2^\ell$ for some $\ell\ge 1$ and  $B\cong C_2^{\ell+1}$, or  $A\cong C_4^2\times C_2^\ell$ for some $\ell\ge 0$ and $B\cong C_4\times C_2^{\ell+1}$.
}
\end{question}

Observe that Theorem \ref{thrm:aa} does not make any claim on the index $2$ subgroups $B'$ of $A$ different from the subgroup $B$ in the statement. Among other things, this point is cleared in the next result. (Given an abelian group $A$ and $a\in A$, we denote by $o(a)$ the order of $a$ and by $A_2$ the subgroup $A_2:=\{a\in A\mid a^2=1\}$.)

\begin{theorem}\label{thrm:11}Let $A$ be an abelian group and let $B$ be a subgroup of $A$ having index $2$. Let $c:=1$ when $A$ has exponent $2$ and let $c:=2$ when $A$ has exponent greater than $2$. Then $A$ contains $2^{\frac{|A|}{4}+\frac{|A_2\setminus B|}{2}}$ inverse-closed subsets $S$ with $S\subseteq A\setminus B$. Moreover, one of the following holds:
\begin{enumerate}
\item the number of inverse-closed subsets $S$ of $A\setminus  B$ such that $\Cay(A,S)$  has bipartite Cayley index $c$ is at least $2^{\frac{|A|}{4}+\frac{|A_2\setminus B|}{2}}-2^{\frac{11|A|}{48}+\frac{|A_2\setminus B|}{2}+(\log_2|A|)^2+2}$,
\item $A\cong C_4\times C_2^\ell$ for some $\ell\ge 1$ and  $B\cong C_2^{\ell+1}$, or  $A\cong C_4^2\times C_2^\ell$ for some $\ell\ge 0$ and $B\cong C_4\times C_2^{\ell+1}$.
\end{enumerate}
\end{theorem}
Theorem \ref{thrm:11} shows that the pairs in the statement of Theorem \ref{thrm:aa} are the only exceptional pairs and, more importantly, for any other possible pair $(A,B)$, the number of ``highly symmetric'' subsets (that is, inverse-closed subsets $S\subseteq A\setminus B$ with $\Cay(A,S)$ not having Cayley index $2$) is bounded above by a relatively slow growing function.

From Theorem \ref{thrm:11} we immediately obtain the following analogue of Corollary \ref{cor:1}.

\begin{corollary}\label{cor:11}For every positive real number $\varepsilon>0$, there exists a natural number $n_\varepsilon$ such that, for every  abelian group $A$ of order at least $n_\varepsilon$ and for every subgroup $B$ of $A$ having index $2$ and with $(A,B)$ not one of the pairs in Theorem~$\ref{thrm:aa}$ (or in Theorem~$\ref{thrm:11}~(2)$), we have
$$\frac{|\{S\mid S\subseteq A\setminus B, S=S^{-1}, \Cay(A,S) \,\mathrm{\,has\, Cayley\, index\, }c\}|}{|\{S\mid S\subseteq A\setminus B, S=S^{-1}\}|}\ge 1-\varepsilon,$$
where $c:=1$ when $A$ has exponent $2$ and  $c:=2$ when $A$ has exponent greater than $2$.
\end{corollary}
\begin{table}[!ht]
\begin{tabular}{|c|c|c|}\hline
$A$&$B$&bipartite Cayley index\\\hline
$C_4\times C_2^{\ell}$&$C_2^{\ell+1}$&not known for $\ell\ge 4$\\
$C_4\times C_4\times C_2^{\ell}$&$C_4\times C_2^{\ell+1}$&not known for $\ell\ge 2$\\

$C_2\times C_2\times C_2$ &$C_2\times C_2$&$6$\\
$C_2\times C_2\times C_2\times C_2$ &$C_2\times C_2\times C_2$&$24$\\
$C_2\times C_2\times C_2\times C_2\times C_2$ &$C_2\times C_2\times C_2\times C_2$&$72$\\
$C_2\times C_2\times C_2\times C_2\times C_2\times C_2$ &$C_2\times C_2\times C_2\times C_2\times C_2$&$4$\\

$C_2\times C_4$ &$C_4$&$6$\\
$C_2\times C_4$ &$C_2\times C_2$&$16$\\

$C_2\times C_8$&$C_2\times C_4$&$16$\\

$C_4\times C_4$&$C_4\times C_2$&$24$\\
$C_4\times C_2\times C_2$& $C_2\times C_2\times C_2$&$768$\\
$C_4\times C_2\times C_2$& $C_4\times C_2$&$24$\\

$C_3\times C_6$&$C_3\times C_3$&$8$\\

$C_2\times C_{12}$&$C_2\times C_6$&$4$\\

$C_2\times C_2\times C_6$&$C_2\times C_6$&$4$\\

$C_4\times C_8$& $C_4\times C_4$&$4$\\
$C_4\times C_8$& $C_2\times C_8$&$4$\\

$C_2\times C_2\times C_8$&$C_2\times C_2\times C_4$&$12$\\

$C_2\times C_4\times C_4$&$C_4\times C_4$&$12$\\
$C_2\times C_4\times C_4$&$C_2\times C_2\times C_4$&$128$\\

$C_2\times C_2\times C_2\times C_4$&$C_2\times C_2\times C_2\times C_2$&$786\, 432$\\
$C_2\times C_2\times C_2\times C_4$&$C_2\times C_2\times C_4$&$72$\\

$C_3\times C_{12}$&$C_3\times C_6$&$4$\\

$C_2\times C_2\times C_{12}$&$C_2\times C_2\times C_6$&$4$\\

$C_3\times C_3\times C_6$&$C_3\times C_3\times C_3$&$12$\\

$C_2\times C_2\times C_2\times C_8$&$C_2\times C_2\times C_2\times C_4$&$8$\\
$C_4\times C_4\times C_4$& $C_2\times C_4\times C_4$&$4$\\
$C_2\times C_2\times C_2\times C_2 \times C_4$&$C_2\times C_2\times C_2\times C_4$&$4$\\\hline
\end{tabular}\caption{Abelian groups and their index $2$ subgroups not admitting a bipartite Cayley graph with Cayley index $2$}\label{table:2}
\end{table}

Exactly as for Corollary \ref{cor:1}, Corollary \ref{cor:11}  says that, when the order of an abelian group $A$ is even and sufficiently large, most bipartite Cayley graphs over $A$  have bipartite Cayley index $2$,  aside from the two exceptional pairs described in Theorem~\ref{thrm:aa}.

Since the estimate in Theorem \ref{thrm:11} is rather explicit, we also obtain the following corollary.
\begin{corollary}\label{cor:22}Let $A$ be an abelian group and let $B$ be a subgroup of $A$ having index $2$. Let $c:=1$ when $A$ has exponent $2$ and let $c:=2$ when $A$ has exponent greater than $2$. Then, either there exists an inverse-closed subset $S$ of $A\setminus B$ such that $\Cay(A,S)$ has Cayley index $c$ or the pair $(A,B)$ is in Table~$\ref{table:2}$.
\end{corollary}
In the third column of Table~\ref{table:2}, we have computed $c(A,B)$, except for the two infinite pairs arising from Theorem~\ref{thrm:aa}.

We also prove the following unlabelled version of Theorem~\ref{thrm:11}.
\begin{theorem}\label{thrm:unlabelled2}Let $A$ be an abelian group  and let $B$ be a subgroup of $A$ having index $2$. Suppose that $(A,B)$ is not one of the pairs in Theorem~$\ref{thrm:aa}$. Then, the number of bipartite Cayley graphs (up to graph-isomorphism) over $A$ with bipartition $\{B,A\setminus B\}$ is $2^{\frac{|A|}{4}+\frac{|A_2\setminus B|}{2}+o(|A|)}$.
\end{theorem}

Based on the work in this paper and on some computer computations, we dare to make the following two conjectures.
\begin{conjecture}\label{conj:1}
{\rm There exists a positive integer $n$ such that, if $G$ is a group of order at least $n$ and  $B$ is a subgroup of $G$ having index $2$, then $\vec{c}(G,B)=1$.}
\end{conjecture}

\begin{conjecture}\label{conj:2}
{\rm There exists a positive integer $n$ such that, if $G$ is a group of order at least $n$ and  $B$ is a subgroup of $G$ having index $2$, then either
\begin{itemize}
\item $c(G,B)=1$, or
\item there exists $\alpha\in \Aut(G)$ with $\alpha\ne 1$, $B^\alpha=B$ and $G=B\cup \{g\in G\mid g^\alpha=g\}\cup \{g\in G\mid g^\alpha=g^{-1}\}$.
\end{itemize}}
\end{conjecture}
Clearly, the groups satisfying the second condition in Conjecture \ref{conj:2} include all abelian groups.
At the time of this writing, we are not sure if the groups satisfying the second condition in Conjecture~\ref{conj:2} might have a meaningful and useful classification. However, we observe that the work of Fitzpatrick, Hegarty, Liebeck and MacHale~\cite{Fitz,HeMa,LMacH} on groups admitting automorphisms inverting many elements seems to be relevant.

\section{Preliminary facts}\label{s: 2}

In what follows we use repeatedly the following facts.
\begin{enumerate}
\item Let $X$ be a finite group. Since a chain of subgroups of $X$ has length at most $\lfloor\log_2|X|\rfloor$, $X$ has a generating set of cardinality at most $\lfloor \log_2|X|\rfloor\le \log_2|X|$.

\item Any automorphism of $X$ is uniquely determined by the images of the elements of a generating set for $X$. Therefore $|\Aut(X)|\le |X|^{\lfloor \log_2|X|\rfloor}\le 2^{(\log_2|X|)^2}$.

\item Any subgroup $Y$ of $X$ is determined by a generating set, which has cardinality at most $\lfloor \log_2|Y|\rfloor\le \lfloor \log_2|X|\rfloor$. Therefore $X$ has at most
$|X|^{\lfloor\log_2|X|\rfloor}\le 2^{(\log_2|X|)^2}$ subgroups.

\item Let $A$ be an abelian group. Then $A$ has at most $|A|$ subgroups $H$ with $|H|$ a prime number. Similarly, $A$ has at most $|A|$ subgroups $K$ with $|A:K|$ a prime number.
\item Let $X$ be a finite group, let $Y$ be a subgroup of $X$ of index $2$ and let $Z$ be a proper subgroup of $X$. If $Z\le Y$, then $|Z\setminus Y|=0$. If $Z\nleq Y$, then $|Z\setminus Y|=|Z|/2\le (|X|/2)/2=|X|/4$.  Therefore, in either case, $|Z\setminus Y|\le |X|/4$.
\end{enumerate}

\section{Existence and asymptotic enumeration of bipartite Cayley digraphs}\label{sec: 33}

\begin{lemma}\label{l:3}
Let $A$ be an abelian group and let $B$ be a subgroup of $A$ having index $2$. The number of subsets $S$ of $A\setminus B$ with $\langle S\rangle $ a proper subgroup of $A$ is at most $2^{\frac{|A|}{4}+\log_2|A|}$.
\end{lemma}
\begin{proof}
Set $N:=|\{
S\subseteq A\setminus  B\mid
\langle S\rangle < A\}|$.
Clearly,
\begin{align*}
\{S\subseteq A\setminus B\mid \langle S\rangle<A\}&=\bigcup_{\substack{C<A\\|A:C|\textrm{ prime}}}\{S\subseteq A\setminus B\mid \langle S\rangle\le C \}.
\end{align*}
Since $\{S\subseteq A\setminus B\mid \langle S\rangle \le C\}=\{S\mid S\subseteq C\setminus (C\cap  B)\}$, we have
\begin{align*}
N&
\le
\sum_{\substack{C<A\\|A:C|\textrm{ prime}}}
|\{S\mid S\subseteq C\setminus (C\cap B)\}|\le
\sum_{\substack{C<A\\|A:C|\textrm{ prime}}}
2^{\frac{|A|}{4}}\le
2^{\frac{|A|}{4}+\log_2|A|},
\end{align*}
where in the second and in the third inequality we used the facts listed in Section~\ref{s: 2}.
\end{proof}

\begin{lemma}\label{l:1}Let $A$ be a group, let $B$ be a subgroup of $A$ having index $2$ and let $\alpha$ be a non-identity automorphism of $A$ with $B^\alpha=B$. The number of subsets $S$ of $A\setminus B$ with $S^\alpha=S$ is at most $2^{\frac{3|A|}{8}}$.
\end{lemma}
\begin{proof}
Since $B$ is $\alpha$-invariant, so is $A\setminus B$. Let $O_1,\ldots,O_\ell$ be the orbits of $\langle\alpha\rangle$ on $A\setminus B$. If $S\subseteq A\setminus B$ is $\alpha$-invariant, then $S$ is a union of some of $O_1,\ldots,O_\ell$ and hence
\begin{equation}\label{eq:1}
|\{S\subseteq A\setminus B\mid S^\alpha=S\}|=2^\ell.
\end{equation}

The orbits of $\langle\alpha\rangle$ on $A$ of cardinality one correspond exactly to the elements of $\cent A{\alpha}:=\{a\in A\mid a^\alpha=a\}$, whereas the orbits of $\langle\alpha\rangle$ on $A\setminus \cent A \alpha$ have cardinality at least $2$.  Now, observing that $|\cent A\alpha|\le |A|/2$ and that
\[
|\cent A\alpha\cap(A\setminus B)|=
\begin{cases}
0&\textrm{ when }\cent A\alpha\le B,\\
|\cent A \alpha \cap B|=|\cent A\alpha|/2&\textrm{ when }\cent A\alpha\nleq B,
\end{cases}
\] we get
\begin{align}\label{eq:2}
\ell&\le
|\cent A{\alpha}\cap (A\setminus B)|
+
\frac{
|(A\setminus B)\setminus (\cent A{\alpha}\cap (A\setminus B))|}{2}
=\frac{|\cent A{\alpha}\cap (A\setminus B)|}{2}+\frac{|A\setminus B|}{2}\\\nonumber
&= \frac{|\cent A{\alpha}\cap (A\setminus B)|}{2}+\frac{|A|}{4}\le \frac{|\cent A\alpha|}{4}+\frac{|A|}{4}\le \frac{|A|/2}{4}+\frac{|A|}{4}=\frac{3}{8}|A|.
\end{align}

The proof now follows from \eqref{eq:1} and \eqref{eq:2}.
\end{proof}

\begin{lemma}\label{l:2}Let $A$ be a group, let $B$ be a subgroup of $A$ having index $2$ and let $H$ and $K$ be subgroups of $A$ with $1<H\le K<A$ and $H\le B$. The number of subsets $S$ of $A\setminus B$ such that $S\setminus K$ is a union of $H$-cosets is at most $2^{\frac{3|A|}{8}}$.
\end{lemma}
\begin{proof}
Observe that $A\setminus (K\cup B)$ is a union of $H$-cosets because $H\le K\cap B$.
Set $$\mathcal{N}:=\{
S\subseteq A\setminus  B\mid
S\setminus K \textrm{ is a union of }H\textrm{-cosets}
\}.$$ If $S\in \mathcal{N}$, then $S\cap K$ is an arbitrary subset of $K\setminus B$ and hence we have $2^{|K\setminus B|}$ choices for $S\cap K$. From this it follows
\begin{align*}
|\mathcal{N}|&=2^{|K\setminus B|+\frac{|(A\setminus B)\setminus (K\setminus B)|}{|H|}}
\le 2^{|K\setminus B|+\frac{|(A\setminus B)\setminus (K\setminus B)|}{2}}=
 2^{|K\setminus B|+\frac{|A\setminus B|}{2}-\frac{|K\setminus B|}{2}}
\\\nonumber
&\le 2^{\frac{|K\setminus B|}{2}+\frac{|A\setminus B|}{2}}
\le 2^{\frac{|K|}{4}+\frac{|A|}{4}}\le 2^{\frac{|A|}{8}+\frac{|A|}{4}}=2^{\frac{3|A|}{8}}.\qedhere
\end{align*}
\end{proof}

\begin{proof}[Proof of Theorem $\ref{thrm:1}$]We partition the set $2^{A\setminus B}:=\{S\mid S\subseteq A\setminus B\}$ in (not necessarily disjoint)  subsets:
\begin{align*}
\mathcal{A}_1:=&\{S\in 2^{A\setminus B}\mid \langle S\rangle <A\},\\
\mathcal{A}_2:=&\{S\in 2^{A\setminus B}\mid \textrm{there exists }\alpha\in \Aut(A) \textrm{ with }\alpha\ne 1, S^\alpha=S\textrm{ and } B^\alpha= B\},\\
\mathcal{A}_3:=&\{S\in 2^{A\setminus B}\mid \textrm{there exist two subgroups }H \textrm{ and }K \textrm{ with }1<H\le K<A, H\le B,\\
&\quad |A:K| \textrm{ and }|H| \textrm{ both prime numbers,}  \textrm{ and }S\setminus K \textrm{ is a union of }H\textrm{-cosets}\},\\
\mathcal{A}_4:=&2^{A\setminus B}\setminus (\mathcal{A}_1\cup\mathcal{A}_2\cup \mathcal{A}_3).
\end{align*}

From Lemma \ref{l:3},
\begin{equation}\label{eq:4}|\mathcal{A}_1|\le 2^{\frac{|A|}{4}+\log_2|A|}.
\end{equation}
Observe now that, if $S\in 2^{A\setminus B}\setminus \mathcal{A}_1$, then $\Cay(A,S)$ is connected and hence $\{ B,A\setminus  B\}$ is the only bipartition of $\Cay(A,S)$. In particular, every automorphism of $\Cay(A,S)$ must preserve the bipartition  $\{ B, A\setminus B\}$.

From Lemma \ref{l:1},
\begin{equation}\label{eq:5}
|\mathcal{A}_2|\le 2^{\frac{3|A|}{8}}(|\Aut(A)|-1)\le 2^{\frac{3|A|}{8}+(\log_2|A|)^2}.
\end{equation}

Since $A$ contains at most $|A|^2$ subgroups $H$ and $K$ with $|H|$ and $|A:K|$ both primes, Lemma \ref{l:2} yields
\begin{equation}\label{eq:6}
|\mathcal{A}_3|\le 2^{\frac{3|A|}{8}+2\log_2|A|}.
\end{equation}

\smallskip

\noindent\textsc{Claim: }For every $S\in \mathcal{A}_4$, $\Cay(A,S)$ is a  bipartite DRR with  bipartition $\{B,A\setminus B\}$.

\smallskip

\noindent  Let $S\in \mathcal{A}_4$, let $\Gamma:=\Cay(A,S)$ and let $G:=\Aut(\Gamma)$. As $S\notin \mathcal{A}_1$, $\Gamma$ is connected, bipartite and $\{B,A\setminus B\}$ is the only bipartition of $\Gamma$.

Since $\Gamma$ is a Cayley digraph over $A$, the group $A$ is embedded in $G$ via its right regular representation. Thus we may identify $A$ as a subgroup of $G$, and we do so. Let $G_1$ be the stabilizer of the vertex $1$ of $\Gamma$. Since  $1\in B$, the group $G_1$ fixes setwise the two parts $B$ and $A\setminus B$ of the bipartition of $\Gamma$, that is, $B^\alpha=B$ for each $\alpha\in G_1$.

Let  $N:=\nor{G_1}A $. Given $\alpha\in N$, we see that $\alpha$ acts as an automorphism on $A$; moreover, $S^\alpha=S$ and $B^\alpha=B$ because $\alpha$ is an automorphism of $\Gamma$ fixing $1$. Thus $N\le \{\alpha\in \Aut(A)\mid S^\alpha=S, B^\alpha=B\}$. Since $S\notin\mathcal{A}_2$, we deduce $N=1$. Thus $A$ is self-normalizing in $G$, that is, $A=\nor G A$. Therefore, since $A$ is abelian, we are in the position to apply~\cite[Theorem $4.2$]{DSV}, see also~\cite[Definition $4.1$]{DSV} for some terminology. We deduce that either $G=A$ and $\Gamma$ is a DRR, or there exist two subgroups $H'$ and $K'$ of $A$ with $1<H'\le K'<A$ and with $S\setminus K'$ a union of $H'$-cosets. We show that the latter possibility yields $S\in\mathcal{A}_3$, contradicting our choice of $S$. Thus, arguing by contradiction, let $H'$ and $K'$ be subgroups of $A$ with $1<H'\le K'\le A$ and with $S\setminus K'$ a union of $H'$-cosets. Let $H\le H'$ and let $K\ge K'$ with $|A:K|$ and $|H|$ both prime numbers. Observe that, since  $H\le H'$ and $K\ge K'$, the set $S\setminus K$ is a union of $H$-cosets. Now, to deduce that $S\in\mathcal{S}_3$, it suffices to show that $H\le B$. Since $\Gamma$ is connected, $S\nsubseteq K$ and hence there exists $s\in S\setminus K$.  Since $s\in S\setminus K$, we have $sH\subseteq S$; therefore, for every $h\in H$, we have $sh\in A\setminus B$. As $s\in A\setminus B$ and $|A:B|=2$, this implies $h\in B$, for every $h\in H$, that is,  $H\le B$.  Therefore $S\in \mathcal{A}_3$, a contradiction. Our claim is now proven.~$_\blacksquare$

\smallskip

Now, the proof  follows immediately from the previous claim,~\eqref{eq:4},~\eqref{eq:5} and~\eqref{eq:6}, and a computation.
\end{proof}

\begin{proof}[Proof of Corollary~$\ref{cor:1new}$]
Let
\begin{align*}
\mathcal{B}&:=\{B\mid B\le A,|A:B|=2\},\\
\mathcal{S}&:=\{S\subseteq A\mid \Cay(A,S)\textrm{ is a bipartite Cayley
digraph}\},\\
\mathcal{D}&:=\{S\subseteq A\mid \Cay(A,S)\textrm{ is a bipartite DRR}\},\\
\mathcal{S}_B&:=\{S\subseteq A\setminus B\mid \Cay(A,S)\textrm{ is a bipartite Cayley digraph with bipartition }\{B,A\setminus B\}\},\\
\mathcal{D}_B&:=\{S\subseteq A\setminus B\mid \Cay(A,S)\textrm{ is a bipartite DRR with bipartition }\{B,A\setminus B\}\}.
\end{align*}
We aim to prove that $|\mathcal{D}|/|\mathcal{S}|\to 1$ as $|A|\to \infty$. Observe that
$$\mathcal{S}=\bigcup_{B\in\mathcal{B}}\mathcal{S}_B,\,\,\mathcal{D}=\bigcup_{B\in \mathcal{B}}\mathcal{D}_B.$$
For $B\in \mathcal{B}$, we have $|\mathcal{S}|\le |\mathcal{B}||\mathcal{S}_B|=|\mathcal{B}|2^{\frac{|A|}{2}}$. Moreover, using the inclusion-exclusion principle, we have
$$|\mathcal{D}|\ge \sum_{B\in\mathcal{B}}|\mathcal{D}_B|-\frac{1}{2}\sum_{\substack{B_1,B_2\in \mathcal{B}\\B_1\ne B_2}}|\mathcal{D}_{B_1}\cap \mathcal{D}_{B_2}|.$$
If $S\in\mathcal{D}_{B_1}\cap\mathcal{D}_{B_2}$, then $S\subseteq (A\setminus B_1)\cap (A\setminus B_2)=A\setminus (B_1\cup B_2)$ and hence we have at most $2^{\frac{|A|}{4}}$ possibilities for $S$. Therefore $|\mathcal{D}_{B_1}\cap\mathcal{D}_{B_2}|\le 2^{\frac{|A|}{4}}$. Using Theorem~\ref{thrm:1}, we get
$$|\mathcal{D}|\ge |\mathcal{B}|(2^{\frac{|A|}{2}}-3\cdot 2^{\frac{3|A|}{8}+(\log_2|A|)^2})-\frac{(|\mathcal{B}|-1)|\mathcal{B}|}{2}2^{\frac{|A|}{4}}.$$
Thus
$$\frac{|\mathcal{D}|}{|\mathcal{S}|}\ge \frac{|\mathcal{D}|}{|\mathcal{B}|2^{\frac{|A|}{2}}}\ge 1-3\cdot 2^{-\frac{|A|}{8}+(\log_2|A|)^2}-\frac{(|\mathcal{B}|-1)}{2}2^{-\frac{|A|}{4}}\to 1,$$
as $|A|\to\infty$.
\end{proof}

\begin{proof}[Proof of Corollary $\ref{cor:2}$]
Let $A$ be an abelian group and let $B$ be a subgroup of $A$ having index $2$. If $|A|\ge 744$, then a computation shows that $|A|/2>3|A|/8+(\log_2|A|)^2+2$ and hence, by Theorem \ref{thrm:1}, there exists a subset $S\subseteq A\setminus B$ with $\Cay(A,S) $ a DRR.

Suppose then $|A|<744$. In this case the proof follows with the invaluable help of the computer algebra system \texttt{magma}~\cite{magma}. Except in the case $A=C_2^6$ all the computations are straightforward. We give some details of these computations. Except for the pairs listed in Table \ref{table:1}, we generate at random $10\,000$ subsets $S$ of $A\setminus B$ and we check whether $\Cay(A,S)$ is a DRR: in all cases, we find a DRR among our digraphs. When $(A,B)$ is one of the pairs in Table \ref{table:1} with $A\neq C_2^6$, we construct all subsets $S$ of $A\setminus B$ and we compute $|\Aut(\Cay(A,S)):A|$; therefore we compute $\vec{c}(A,B)$ by brute force. When $A=C_2^6$ and  $B$ has index $2$ in $A$ this naive approach does not work  because we have  $2^{ |A\setminus B|}=2^{32}$ subsets to check.

Suppose $A=C_2^6$ and  $B$ has index $2$ in $A$. We aim to prove that $\vec{c}(A,B)=4$. We identify $A$ with the $6$-dimensional vector space $\mathbb{F}_2^6$ of column vectors over the field $\mathbb{F}_2$ of size $2$ and we identify $B$ with the hyperplane of $A$ with equation $x_1+x_2+x_3+x_4+x_5+x_6=0$. Let $S$ be a subset of $A\setminus B$ with $\vec{c}(A,B)=|\Aut(\Cay(A,S)):A|$.  Write $\Gamma:=\Cay(A,S)$, $s:=|S|$, $G:=\Aut(A)$ and $H:=\{\alpha\in G\mid B^\alpha=B\}$. Clearly, $G\cong\mathrm{GL}_6(2)$ and $H\cong \mathrm{AGL}_5(2)$.

Replacing $S$ with $(A\setminus B)\setminus S$ if necessary, we may assume that $s=|S|\le|A\setminus B|/2=16$. If $\langle S\rangle<A$, then $\Gamma$ is disconnected and we can apply Table~\ref{table:1} to the elementary abelian $2$-group $\langle S\rangle\cong C_2^\ell$ with $\ell\le 5$. For instance, if $\langle S\rangle\cong C_2^5$, then $|\Aut(\Cay(A,S)):A|\ge |\Aut(\Cay(C_2^5,S))\,\mathrm{wr} \,C_2|/2^6\ge (72\cdot 2^5)^2\cdot 2/2^6=2!\cdot 72^2\cdot 2^4$. Thus, we obtain  $$|\Aut(\Cay(A,S)):A|\ge\min\{
2!\cdot 72^2\cdot 2^4,
4!\cdot 24^4\cdot 2^{10},
8!\cdot 6^8\cdot 2^{18},
16!\cdot 2^{16}\cdot 2^{26}
\}= 165\, 888.$$ It remains to consider the case that $A=\langle S\rangle$, that is, $S$ contains an $\mathbb{F}_2$-basis of $A$.

A computation shows that $H=\mathrm{AGL}_5(2)$ acts transitively on the $\mathbb{F}_2$-basis
 of $A$ contained in $A\setminus B$. Therefore, we may assume that $S$ contains the six canonical vectors $e_1,e_2,e_3,e_4,e_5,e_6$. Write $$\mathcal{B}:=\{e_1,e_2,e_3,e_4,e_5,e_6\}\quad\textrm{ and }K:=\{h\in H\mid \mathcal{B}^h=\mathcal{B}\}.$$ Observe that $K\cong \mathrm{Sym}(6)$ is the group of monomial matrices.

Now, we may write $S=\mathcal{B}\cup T$, for some subset $T$ of $A\setminus (B\cup\mathcal{B})$ of cardinality at most $16-6=10$. If $T=\emptyset$, then $|\Aut(\Cay(A,S)):A|\ge |K|=6!=720$, because the group of monomial matrices $K$ fixes setwise $S$ and hence is a group of automorphisms of $\Cay(A,S)$. Suppose that $T\ne\emptyset$. A computation reveals that $K$ has two orbits on the vectors in $A\setminus (B\cup \mathcal{B})$, with representatives $e_1+e_2+e_3$ and $e_1+e_2+e_3+e_4+e_5$. Write
$\mathcal{B}_1:=\mathcal{B}\cup \{e_1+e_2+e_3\}$ and $\mathcal{B}_2:=\mathcal{B}\cup \{e_1+e_2+e_3+e_4+e_5\}$. Replacing $S$ by a suitable $K$-conjugate if necessary, we may assume that $S$ contains either $\mathcal{B}_1$ or $\mathcal{B}_2$. Therefore, we have two cases to consider
\begin{itemize}
\item $S=\mathcal{B}_1\cup T_1$, for some subset $T_1$ of $A\setminus (B\cup\mathcal{B}_1)$ of cardinality at most $16-7=9$,
\item $S=\mathcal{B}_2\cup T_2$, for some subset $T_2$ of $A\setminus (B\cup\mathcal{B}_2)$ of cardinality at most $16-7=9$.
\end{itemize}
Since $|A\setminus(B\cup\mathcal{B}_i)|=25$, the number of possibilities for $S$ is
$$2\left(
{25\choose 0}+
{25\choose 1}+
{25\choose 2}+
{25\choose 3}+
{25\choose 4}+
{25\choose 5}+
{25\choose 6}+
{25\choose 7}+
{25\choose 8}+
{25\choose 9}
\right)=7\, 701\, 512.$$
This number is within computational reach, therefore we have generated all possible subsets $S$ as above and we have checked that the minimum for $|\Aut(\Cay(A,S)):A|$ is $4$.
\end{proof}

An \textit{unlabeled} digraph is simply an equivalence class of digraphs under the relation ``being digraph-isomorphic to". In the proof of Theorem~\ref{thrm:unlabelled1}, we identity a representative with its class.
\begin{proof}[Proof of Theorem~$\ref{thrm:unlabelled1}$]
For the proof, we let $\mathrm{DRR}(A,B)$ denote the set of unlabelled bipartite DRRs over $A$ with bipartition $\{B,A\setminus B\}$ and let $2^{A\setminus B}_{\mathrm{DRR}}$ be the collection of the subsets $S$ of $A\setminus B$ with $\Cay(A,S)$ a DRR.

Let $S_1$ and $S_2$ be in $2^{A\setminus B}_{\mathrm{DRR}}$ and let $\Gamma_1 := \Cay(A, S_1)$ and $\Gamma_2 := \Cay(A, S_2 )$. Suppose that $\Gamma_1\cong\Gamma_2$ and let $\varphi$ be a digraph isomorphism from $\Gamma_1$ to $\Gamma_2$. Without loss of generality, we may assume that $1^\varphi=1$. Note that $\varphi$ induces a group automorphism from $\Aut(\Gamma_1)=A$ to $\Aut(\Gamma_2)=A$. In particular, $\varphi\in \Aut(A)$ and $S_1$ and $S_2$ are conjugate via an element of $\Aut(A)$. This shows that $$|\mathrm{DRR}(A,B)|\geq \frac{|2^{A\setminus B}_{\mathrm{DRR}}|}{|\Aut(A)|}.$$ By Theorem~\ref{thrm:1}, we have $$|2^{A\setminus B}_{\mathrm{DRR}}|\geq 2^{\frac{|A|}{2}}-3\cdot 2^{\frac{3|A|}{|8|}+(\log_2|A|)^2}.$$ Since $|\Aut(A)|\leq  2^{(\log_2|A|)^2}$, it follows that
$$|\mathrm{DRR}(A,B)|\geq 2^{\frac{|A|}{2}-(\log_2|A|)^2}-3\cdot 2^{\frac{3|A|}{|8|}}.$$
In particular, this proves the first part of the theorem.

Let $\mathrm{UCD}(A,B)$ denote the set of unlabeled bipartite Cayley digraphs on $A$ with bipartition $\{B,A\setminus B\}$ that are not DRRs. Clearly, $|\mathrm{DRR}(A,B)|+|\mathrm{UCD}(A,B)|$ is the number of unlabelled bipartite Cayley graphs on $A$ with bipartition $\{B,A\setminus B\}$. Note that
$$\frac{|\mathrm{DRR}(A,B)|}{|\mathrm{DRR}(A,B)|+|\mathrm{UCD}(A,B)|}=1-\frac{|\mathrm{UCD}(A,B)|}{|\mathrm{DRR}(A,B)|+|\mathrm{UCD}(A,B)|}\geq 1-\frac{|\mathrm{UCD}(A,B)|}{|\mathrm{DRR}(A,B)|}.$$
By Theorem~\ref{thrm:1}, we have $|\mathrm{UCD}(A,B)|\leq 2^{\frac{3|A|}{8}+o(|A|)}$ and thus
$$\frac{|\mathrm{UCD}(A,B)|}{|\mathrm{DRR}(A,B)|}\to 0,$$
as $|A|\to\infty$. This completes the proof of the second part of the theorem.
\end{proof}

\section{Existence and asymptotic enumeration of bipartite Cayley graphs}\label{sec: 3}
Given a group $A$, we write $A_2$ for the subset  $\{a\in A\mid o(a)\le 2\}$. When $A$ is abelian, $A_2$ is a subgroup of $A$.

\begin{lemma}\label{l:4}
Let $A$ be an abelian group and let $B$ be a subgroup of $A$ having index $2$. Then
\[
|\{S\subseteq A\setminus B\mid S=S^{-1}\}|=
2^{\frac{|A|}{4}+\frac{|A_2\setminus B|}{2}}=\begin{cases}
2^{\frac{|A|}{4}}&\textit{if } A_2\le B,\\
2^{\frac{|A|}{4}+\frac{|A_2|}{4}}&\textrm{if }A_2\nleq B.
\end{cases}
\]
\end{lemma}
\begin{proof}
We may partition the inverse-closed subsets $S$ of $A\setminus B$ in two parts $S_1:=S\cap A_2$ and $S_2:=S\setminus A_2$. Clearly, $S_1$ is an arbitrary subset of $A_2\setminus B$, whereas since none of the elements in $A\setminus A_2$ is an involution, the elements in $S_2$ come in pairs: each element paired up to its inverse. Therefore, for $S_1$ we have $2^{|A_2\setminus B|}$ choices and  for $S_2$ we have
$2^{\frac{|A\setminus (A_2\cup B)|}{2}}=2^{\frac{|(A\setminus B)\setminus (A_2\setminus B)|}{2}}=2^{\frac{|A\setminus B|}{2}-\frac{|A_2\setminus B|}{2}}$ choices. Now, the proof follows.
\end{proof}

\begin{lemma}\label{l:3new}
Let $A$ be an abelian group and let $B$ be a subgroup of $A$ having index $2$. The number of inverse-closed subsets $S$ of $A\setminus B$ with $\langle S\rangle $ a proper subgroup of $A$ is at most $2^{\frac{|A|}{8}+\frac{|A_2\setminus B|}{2}+\log_2|A|}$.
\end{lemma}
\begin{proof}
Set $N:=|\{
S\subseteq A\setminus  B\mid
S=S^{-1},\langle S\rangle < A\}|$.
Clearly,
\begin{align*}
\{S\subseteq A\setminus B\mid S=S^{-1},\langle S\rangle<A\}&=\bigcup_{\substack{C<A\\|A:C|\textrm{ prime}}}\{S\subseteq A\setminus B\mid S=S^{-1},\langle S\rangle\le C \}.
\end{align*}
Since $\{S\subseteq A\setminus B\mid S=S^{-1}, \langle S\rangle \le C\}\subseteq \{S\mid S=S^{-1},S\subseteq C\setminus (C\cap  B)\}$, using Lemma \ref{l:4} we have
\begin{align*}
N&
\le
\sum_{\substack{C<A\\|A:C|\textrm{ prime}}}
|\{S\mid S=S^{-1},S\subseteq C\setminus (C\cap B)\}|=
\sum_{\substack{C<A\\|A:C|\textrm{ prime}}}
2^{\frac{|C|}{4}+\frac{|C_2\setminus B|}{2}}\\
&\le \sum_{\substack{C<A\\|A:C|\textrm{ prime}}}
2^{\frac{|A|}{8}+\frac{|A_2\setminus B|}{2}}\le
2^{\frac{|A|}{8}+\frac{|A_2\setminus B|}{2}+\log_2|A|}.\qedhere
\end{align*}
\end{proof}

\begin{example}\label{example:1}
{\rm
Let $\ell$ be a positive integer with $\ell\ge 1$, let $A:=\langle x\rangle\times \langle y_1\rangle\times\langle y_2\rangle\times \cdots \times\langle y_\ell\rangle$ with $o(x)=4$ and $o(y_i)=2$, for each $i\in \{1,\ldots,\ell\}$, and  let $B:=\langle x^2,y_1,y_2,\ldots,y_\ell\rangle$. Here we show that no bipartite Cayley graph over $A$ with bipartition $\{B,A\setminus B\}$ has Cayley index $2$, that is, $c(A,B)>2$.

Let $\alpha:A\to A$ be the automorphism of $A$ defined on the generators by
\begin{align*}
x^\alpha&=x^{-1},\\
y_1^\alpha&=x^2y_1, \textrm{ and }\\
y_i^\alpha&=y_i, \textrm{ for each }i\in \{2,\ldots,\ell\}.
\end{align*}
Observe that $\alpha$ is a non-identity automorphism and $\alpha\ne \iota$. Moreover, $$(xy_1)^\alpha=x^\alpha y_1^\alpha=x^{-1}x^2y_1=xy_1.$$  Therefore $\alpha$ fixes each element in the subgroup $$T_1:=\langle xy_1,y_2,\ldots,y_\ell\rangle$$ and  $\alpha$ inverts each element in the subgroup $$T_{-1}:=\langle x,y_2,\ldots,y_\ell\rangle.$$ Clearly, $$T_1\cup T_{-1}\supseteq A\setminus B$$ and hence $\{a,a^{-1}\}^\alpha=\{a,a^{-1}\}$,  for every $a\in A\setminus B$. In particular, for every inverse-closed subset $S\subseteq A\setminus B$, $\alpha$ is a non-identity graph automorphism of $\Cay(A,S)$. Thus $\Aut(\Cay(A,S))\ge\langle A,\iota,\alpha\rangle$  and $c(A,B)\ge 4$.}
\end{example}

\begin{example}\label{example:2}
{\rm Let $\ell$ be a non-negative integer, let $A:=\langle x_1\rangle\times\langle x_2\rangle\times \langle y_1\rangle\times\langle y_2\rangle\times \cdots \times\langle y_\ell\rangle$ with $o(x_1)=o(x_2)=4$ and $o(y_i)=2$, for each $i\in \{1,\ldots,\ell\}$, and  let $B:=\langle x_1^2,x_2,y_1,\ldots,y_\ell\rangle$. Here we show  that no bipartite Cayley graph over $A$ with bipartition $\{B,A\setminus B\}$ has Cayley index $2$, that is, $c(A,B)>2$.

Let $\alpha:A\to A$ be the automorphism of $A$ defined on the generators by
\begin{align*}
x_1^\alpha&=x_1,\\
x_2^\alpha&=x_1^2x_2^{-1}, \textrm{ and}\\
y_i^\alpha&=y_i, \textrm{ for each }i\in \{1,\ldots,\ell\}.
\end{align*}
Observe that $\alpha$ is a non-identity automorphism and $\alpha\ne \iota$. Moreover,
$$(x_1x_2)^\alpha=x_1^\alpha x_2^\alpha=x_1 x_1^{2}x_2^{-1}=x_1^{-1}x_2^{-1}=(x_1x_2)^{-1}.$$
Therefore $\alpha$ fixes each element in the subgroup $$T_1:=\langle x_1,x_2^2,y_1,y_2,\ldots,y_\ell\rangle$$ and $\alpha$ inverts each element in the subgroup $$T_{-1}:=\langle x_1^2,x_1x_2,y_1,\ldots,y_\ell\rangle.$$ Clearly, $$T_1\cup T_{-1}\supseteq A\setminus B$$ and hence $\{a,a^{-1}\}^\alpha=\{a,a^{-1}\}$,  for every $a\in A\setminus B$. In particular, for every inverse-closed subset $S\subseteq A\setminus B$, $\alpha$ is a non-identity graph automorphism of $\Cay(A,S)$. Thus $\Aut(\Cay(A,S))\ge \langle A,\iota,\alpha\rangle$ and $c(A,B)\ge 4$.
}
\end{example}

\begin{proof}[Proof of Theorem~$\ref{thrm:aa}$]
Let $\ell$ be a positive integer with $\ell\ge 1$, let $A:=\langle x\rangle\times \langle y_1\rangle\times\langle y_2\rangle\times \cdots \times\langle y_\ell\rangle$ with $o(x)=4$ and $o(y_i)=2$, for each $i\in \{1,\ldots,\ell\}$. Thus $A\cong C_4\times C_2^\ell$. Let $B$ be a subgroup of $A$ with $B\cong C_2^{\ell+1}$.
It is an easy computation to see that $\Aut(A)$ has only two orbits in its action on the subgroups of $A$ having index $2$; moreover, for the two orbits we may take representatives $\langle x^2,y_1,y_2,\ldots,y_\ell\rangle\cong C_2^{\ell+1}$ and $\langle x,y_1,y_2,\ldots,y_{\ell-1}\rangle\cong C_4\times C_2^{\ell-1}$. Therefore, $B=\langle x^2,y_1,y_2,\ldots,y_\ell\rangle$ and, in this case, the proof follows from Example~\ref{example:1}.

Let $\ell$ be a non-negative integer, let $A:=\langle x_1\rangle\times\langle x_2\rangle\times \langle y_1\rangle\times\langle y_2\rangle\times \cdots \times\langle y_\ell\rangle$ with $o(x_1)=o(x_2)=4$ and $o(y_i)=2$, for each $i\in \{1,\ldots,\ell\}$. Thus $A\cong C_4^2\times C_2^\ell$. Let $B$ be a subgroup of $A$ with $B\cong C_4\times C_2^{\ell+1}$.
When $\ell\ge 1$, the group $\Aut(A)$ has only two orbits in its action on the subgroups of $A$ having index $2$; moreover, for the two orbits we may take representatives $\langle x_1^2,x_2,y_1,y_2,\ldots,y_\ell\rangle\cong C_4\times C_2^{\ell+1}$ and $\langle x_1,x_2,y_1,y_2,\ldots,y_{\ell-1}\rangle\cong C_4^2\times C_2^{\ell-1}$. Therefore, replacing $B$ by a suitable $\Aut(A)$-conjugate, we may assume that $B=\langle x_1^2,x_2,y_1,y_2,\ldots,y_\ell\rangle$. When $\ell=0$, every subgroup of $A$ having index $2$ is isomorphic to $C_4\times C_2$ and hence, again, we may take $B=\langle x_1^2,x_2\rangle$. Now, the proof in this case follows from  Example~\ref{example:2}.
\end{proof}

\begin{lemma}\label{l:21}
Let $A$ be an abelian group not having cardinality a power of $2$, let $B$ be a subgroup of $A$ having index $2$, and let $H$ and $K$ be subgroups of $A$ with $1<H\le K<A$ and $H\le B$. The number of inverse-closed subsets $S$ of $A\setminus B$ such that $S\setminus K$ is a union of $H$-cosets is at most $2^{\frac{11|A|}{48}+\frac{|A_2\setminus B|}{2}}$.
\end{lemma}
\begin{proof}
We subdivide the proof in various cases. For simplicity, we write
$$\mathcal{S}:=\{S\subseteq A\setminus B\mid S=S^{-1}, S\setminus K\textrm{ is a union of }H\textrm{-cosets}\}\quad\textrm{and}\quad s:=|\mathcal{S}|.$$

\smallskip

\noindent\textsc{Case 1: }$K\le B$ and $|H|> 2$.

\smallskip

\noindent If $S\in\mathcal{S}$, then $S\setminus K=S$ and hence the whole of $S$ is a union of $H$-cosets. Since $A\setminus B$ has $2^{\frac{|A\setminus B|}{|H|}}$ subsets that are union of $H$-cosets (regardless of whether they are is inverse-closed or not) and since $2^{\frac{|A\setminus B|}{|H|}}=2^{\frac{|A|}{2|H|}}\le 2^{\frac{|A|}{6}}$, we have
\begin{equation}\label{eq:21_1}
s\le 2^{\frac{|A|}{6}}.
\end{equation}

\smallskip

\noindent\textsc{Case 2: }$K\nleq B$ and $|H|> 2$.

\smallskip

\noindent Let $S\in\mathcal{S}$. We may partition $S$ into two parts $S_1:=S\cap K$ and $S_2:=S\setminus K$. Observe that $K\cap  B$ has index $2$ in $K$  because $K\nleq  B$. Therefore, by Lemma \ref{l:4} applied to the group $K$ and to the index $2$ subgroup $B\cap K$, the number of choices for $S_1$ is exactly
$$2^{\frac{|K|}{4}+\frac{|K_2\setminus B|}{2}}.$$

Similarly to Case~1, to obtain an upper bound for the number of choices of  $S_2$, we simply count the number of subsets $S_2$ of  $(A\setminus B)\setminus (K\setminus B)$ that are union of $H$-cosets, regardless of whether $S_2$ is inverse-closed or not. Thus the number of choices for $S_2$  is at most
$$2^{\frac{|(A\setminus B)\setminus (K\setminus B)|}{|H|}}=2^{\frac{|A|}{2|H|}-\frac{|K|}{2|H|}}.$$
Combining the upper bounds for $S_1$ and $S_2$, we have
$$s\le
2^{
\frac{|A|}{2|H|}
-\frac{|K|}{2|H|}+
\frac{|K|}{4}+\frac{|K_2\setminus B|}{2}}=
2^{\frac{|A|}{2|H|}+\frac{|K|}{4}\left(1-\frac{2}{|H|}\right)+\frac{|K_2\setminus B|}{2}}
\le
2^{\frac{|A|}{2|H|}+\frac{|K|}{4}\left(1-\frac{2}{|H|}\right)+\frac{|A_2\setminus B|}{2}}.$$
Using $|H|>2$, we deduce $1/2|H|\le 1/6$ and $1-2/|H|\le 1/3$. Thus
$$s\le 2^{\frac{|A|}{6}+\frac{|K|}{12}+\frac{|A_2\setminus B|}{2}}.$$
Since $|K|\le |A|/2$, we have
\begin{equation}\label{eq:21_3}
s\le 2^{\frac{5|A|}{24}+\frac{|A_2\setminus B|}{2}}.
\end{equation}

\smallskip

For the rest of the proof we may assume that $|H|=2$.   Write $H:=\langle h\rangle$. We start with a preliminary observation. Let $a\in A$. If $\{a,a^{-1}\}h=\{a,a^{-1}\}$, then $ah=a^{-1}$ and $a^{-1}h=a$, that is, $a^2=h$. Therefore, under the action of right multiplication by $h$, the only pairs $\{a,a^{-1}\}$ that are fixed by $h$ satisfy $a^2=h$.

\smallskip

\noindent\textsc{Case 3: }$K\le B$ and $|H|= 2$.

\smallskip

\noindent As in Case 1, all of $S$ is a union of $H$-cosets. Write $$T:=\{a\in A\setminus B\mid a^2=h\}.$$ Observe that $T$ contains only elements having order $4$ and hence $T\cap A_2=\emptyset$. Let $S\in\mathcal{S}$. The elements in $S\cap T$ come in pairs: each element $x$ paired up to $x^{-1}$. The elements in $S\cap (A_2\setminus B)$ also come in pairs: each element $x$ paired up with $xh$. The elements in $S\setminus (T\cup (A_2\setminus B))$ come in fours: each element $x$ comes along with $x,x^{-1},xh$ and $x^{-1}h$. Thus, from our preliminary observation, we have
\begin{align}\label{eq:my}
s&=2^{\frac{|A_2\setminus B|}{2}+\frac{|T|}{2}+\frac{|(A\setminus B)\setminus (T\cup (A_2\setminus B))|}{4}}=2^{\frac{|A_2\setminus B|}{2}+\frac{|T|}{2}+\frac{|A\setminus B|}{4}-\frac{|A_2\setminus B|}{4}-\frac{|T|}{4}+\frac{|T\cap A_2|}{4}}\\\nonumber
&=
2^{\frac{|A\setminus B|}{4}+\left(\frac{|T|}{2}-\frac{|T|}{4}\right)+\left(\frac{|A_2\setminus B|}{2}-\frac{|A_2\setminus B|}{4}\right)}=
2^{\frac{|A|}{8}+\frac{|T|}{4}+\frac{|A_2\setminus B|}{4}}.
\end{align}

If  $T=\emptyset$, then~\eqref{eq:my} gives $$s= 2^{\frac{|A|}{8}+\frac{|A_2\setminus B|}{4}}.$$
Suppose $T\ne\emptyset$ and let $a_0\in T$. An easy computation yields $$T=a_0(A_2\cap B)=\{a_0c\mid c\in A_2\cap B\}.$$ Thus, $|T|=|A_2\cap B|$. Since $A$ is not a $2$-group and $|A:B|=2$, we have $|A:A_2\cap B|=|A:B||B:A_2\cap B|=2| B:A_2\cap  B|\ge 6$ and hence~\eqref{eq:my} gives
$$s=2^{\frac{|A|}{8}+\frac{|A_2\cap B|}{4}+\frac{|A_2\setminus B|}{4}}\le 2^{\frac{|A|}{8}+\frac{|A|}{24}+\frac{|A_2\setminus B|}{4}}\le
2^{\frac{|A|}{6}+\frac{|A_2\setminus B|}{4}}.$$

Summing up, we have shown
\begin{equation}\label{eq:21_2}
s\le 2^{\frac{|A|}{6}+\frac{|A_2\setminus B|}{4}}.
\end{equation}

\smallskip

\noindent\textsc{Case 4: }$K\nleq B$ and $|H|= 2$.

\smallskip

\noindent We use the ideas in Cases~2 and~3. Write $$T:=\{a\in (A\setminus B)\setminus (K\setminus B)\mid a^2=h\}.$$ Observe that $T$ contains only elements having order $4$ and hence $T\cap A_2=\emptyset$. Write also $$R:=(A_2\setminus B)\setminus (K\setminus B).$$ The sets $T$, $R$ and $K\setminus B$ are mutually disjoint and $R\cup (K_2\setminus B)=A_2\setminus B$.

Let $S\in \mathcal{S}$. By Lemma \ref{l:4} applied to the group $K$ and to the index $2$ subgroup $B\cap K$, the number of choices for $S\cap K$ is exactly
$2^{\frac{|K|}{4}+\frac{|K_2\setminus B|}{2}}.$ The elements in $S\setminus  K$ can be partitioned in three subsets $$S_1:=(S\setminus K)\cap T,\,\,\, S_2:=(S\setminus K)\cap R\, \textrm{   and   }\,S_3:=(S\setminus K)\setminus (S_1\cup S_2).$$ As $R$, $T$ and $S\setminus K$ are inverse-closed and unions of $H$-cosets, so are $S_1$, $S_2$ and $S_3$. Therefore the elements in $S_1$ and in $S_2$ come in pairs (the element $x$ in $S_1$ paired up to $x^{-1}$ and the element $x$ in $S_2$ paired up to $xh$), whereas the elements in $S_3$ come in fours. Thus
\begin{align}\label{eq:MY}
s&=2^{\frac{|(A\setminus B)\setminus (R\cup T\cup(K\setminus  B))|}{4}+\frac{|T|}{2}+\frac{|R|}{2}+\frac{|K|}{4}+\frac{|K_2\setminus B|}{2}}=
2^{\frac{|A\setminus B|}{4}-\frac{|R|}{4}-\frac{|T|}{4}-\frac{|K\setminus B|}{4}+\frac{|T|}{2}+\frac{|R|}{2}+\frac{|K|}{4}+\frac{|K_2\setminus B|}{2}}
\\\nonumber
&=2^{\frac{|A|}{8}+\left(\frac{|R|}{2}-\frac{|R|}{4}\right)+\left(\frac{|T|}{2}-\frac{|T|}{4}\right)-\frac{|K|}{8}+\frac{|K|}{4}+\frac{|K_2\setminus B|}{2}}\\\nonumber
&=
2^{
\frac{|A|}{8}+\frac{|R|}{4}+\frac{|T|}{4}+\frac{|K|}{8}+\frac{|K_2\setminus B|}{2}}\le
2^{
\frac{|A|}{8}+\frac{|T|}{4}+\frac{|K|}{8}+\frac{|A_2\setminus B|}{2}}.
\end{align}

If  $T=\emptyset$, then $$s\le
2^{\frac{|A|}{8}+\frac{|K|}{8}+\frac{|A_2\setminus B|}{2}}\le 2^{\frac{|A|}{8}+\frac{|A|}{16}+\frac{|A_2\setminus B|}{2}}=2^{\frac{3|A|}{16}+\frac{|A_2\setminus B|}{2}}.$$
Suppose $T\ne\emptyset$ and let $a_0\in T$. An easy computation gives $$T\subseteq  a_0(A_2\cap B)=\{a_0c\mid c\in A_2\cap B\}.$$ Thus, $|T|\le |A_2\cap B|$. Since $A$ is not a $2$-group and $|A:B|=2$, we have $|A_2\cap B|\le |A|/6$ and hence $|T|\le |A|/6$. From~\eqref{eq:MY}, we deduce
\begin{equation}\label{eq:last}
s\le 2^{\frac{|A|}{8}+\frac{|A|}{24}+\frac{|A|}{16}+\frac{|A_2\setminus B|}{2}}=2^{\frac{11|A|}{48}+\frac{|A_2\setminus B|}{2}}.
\end{equation}

Now the proof follows from~\eqref{eq:21_1},~\eqref{eq:21_3},~\eqref{eq:21_2} and~\eqref{eq:last}.
\end{proof}

\begin{lemma}\label{l:22}
Let $A$ be an abelian group of exponent greater than $2$, let $B$ be a subgroup of $A$ having index $2$ and let $\alpha$ be a non-identity automorphism of $A$ with $B^\alpha=B$ and $\alpha\ne \iota$.  Suppose that $(A,B)$ is not one of the pairs in the statement of Theorem~$\ref{thrm:aa}$. Then, the number of inverse-closed subsets $S$ of $A\setminus B$ with $S^\alpha=S$ is at most $2^{\frac{11|A|}{48}+\frac{|A_2\setminus B|}{2}}$.
\end{lemma}
\begin{proof}
Set
\begin{align*}
T_1&:=\{a\in A\mid a^\alpha=a\},\\
T_{-1}&:=\{a\in A\mid a^\alpha=a^{-1}\},\\
\mathcal{S}&:=\{S\subseteq A\setminus  B\mid S=S^{-1}, S^\alpha=S\},\\
s&:=|\mathcal{S}|.
\end{align*}
Observe that $T_1\cap T_{-1}\le A_2$.

We start with a preliminary remark:  given a subset $S\subseteq A\setminus B$, $S$ satisfies $S=S^{-1}$ and $S^\alpha=S$ if and only if $S$ is invariant by the subgroup $\langle \iota,\alpha\rangle$ of $\Aut(A)$.
Now, we divide the proof in two  cases, and each case in various subcases.

\smallskip

\noindent \textsc{\textbf{Case 1:} } $A_2\le B$.

\smallskip

\noindent\textsc{Case 1a: }$T_1\cup T_{-1}\subseteq B$.

\noindent The conditions in Case 1a guarantee that each orbit of $\langle\iota,\alpha\rangle$ on $A\setminus B$ has cardinality at least $4$. Therefore,
\begin{equation*}
s\le 2^{\frac{|A\setminus B|}{4}}=2^{\frac{|A|}{8}}<2^{\frac{11|A|}{48}+\frac{|A_2\setminus B|}{2}}.
\end{equation*}

\smallskip

\noindent \textsc{Case 1b: }$T_1\le B$ and  $T_{-1}\nleq B$, or $T_1\nleq B$ and $T_{-1}\le B$.

\smallskip

\noindent We only deal with the case $T_1\le  B$, $T_{-1}\nleq B$ and $A_2\le B$, the other case is similar.
As $A_2\le B$, $\langle \alpha,\iota\rangle$ has orbits of size at least $2$ on $(A\setminus B)\cap (T_1\cup T_{-1})=T_{-1}\setminus B$ and of size at least $4$ on $(A\setminus B)\setminus (T_1\cup T_{-1})$. Therefore,
\begin{align*}
s&\leq
2^{\frac{|(A\setminus B)\cap (T_1\cup T_{-1})|}{2}+\frac{|(A\setminus B)\setminus (T_1\cup T_{-1})|}{4}}
=
2^{\frac{|A\setminus B|}{4}+\frac{|(A\setminus B)\cap (T_1\cup T_{-1})|}{4}}\\
&=
2^{\frac{|A|}{8}+\frac{|(A\setminus B)\cap (T_1\cup T_{-1})|}{4}}=2^{\frac{|A|}{8}+\frac{|T_{-1}\setminus B|}{4}}
=
2^{\frac{|A|}{8}+\frac{|T_{-1}|}{8}}
\le 2^{\frac{|A|}{8}+\frac{|A|}{16}}\\
&= 2^{\frac{3|A|}{16}}<2^{\frac{11|A|}{48}+\frac{|A_2\setminus B|}{2}}.
\end{align*}

\smallskip

\noindent \textsc{Case 1c: }$T_1\nleq B$ and  $T_{-1}\nleq  B$.

\smallskip

\noindent Observe that $T_1\cap T_{-1}\le A_2\le B$. In particular, $T_1\ne T_{-1}$. We argue as  in the case above but slightly refining the argument. As $A_2\le B$, $\langle \alpha,\iota\rangle$ has orbits of size at least $2$ on $(A\setminus B)\cap (T_1\cup T_{-1})$ and of size at least $4$ on $(A\setminus B)\setminus (T_1\cup T_{-1})$. Therefore,
\begin{align*}
s&\leq 2^{\frac{|(A\setminus B)\cap (T_1\cup T_{-1})|}{2}+\frac{|(A\setminus B)\setminus (T_1\cup T_{-1})|}{4}}=2^{\frac{|A|}{8}+\frac{|(A\setminus B)\cap (T_1\cup T_{-1})|}{4}}=2^{\frac{|A|}{8}+\frac{|T_1\setminus B|}{4}+\frac{|T_{-1}\setminus B|}{4}}\\
&=2^{\frac{|A|}{8}+\frac{|T_1|}{8}+\frac{|T_{-1}|}{8}}.
\end{align*}

If $|A:T_1|\ge 3$ or $|A:T_{-1}|\ge 3$, then
$$s\le 2^{\frac{|A|}{8}+\frac{|A|}{16}+\frac{|A|}{24}}=2^{\frac{11|A|}{48}}=2^{\frac{11|A|}{48}+\frac{|A_2\setminus B|}{2}}.$$

Finally, suppose that $|A:T_1|=|A:T_{-1}|=2$. In particular, as $T_{1}\cap T_{-1}\le A_2$, we deduce $A_2$ has index either $2$ or $4$ in $A$.

Assume first that $|A:A_2|=4$. Then $A_2=T_1\cap T_{-1}$. Moreover, since $T_1/A_2$ and $T_{-1}/A_2$ are two distinct subgroups of $A/A_2$ of order $2$, we deduce $A/A_2\cong C_2\times C_2$. Then $A\cong C_4\times C_4\times C_2^{\ell}$ for some $\ell\ge 0$.  Since $B$ contains $A_2$, we have $B\cong C_4\times C_2^{\ell+1}$. Therefore $(A,B)$ is one of the pairs in the statement of Theorem~\ref{thrm:aa}, contradicting one of the hypotheses of this lemma.

Assume that $|A:A_2|=2$. Therefore $A\cong C_4\times C_2^\ell$, for some $\ell\ge 1$.  (If $\ell=0$, then $A\cong C_4$. However, $C_4$ has a unique subgroup of index $2$, forcing $T_1=T_{-1}=A_2$.)  As $A_2\le B$ and $|B|=|A_2|$, we must have $B=A_2\cong C_2^{\ell+1}$. Therefore $(A,B)$ is one of the pairs in the statement of Theorem~\ref{thrm:aa}, contradicting one of the hypotheses of this lemma.

\smallskip

This concludes the proof of \textbf{Case 1}.

\smallskip

\noindent \textsc{\textbf{Case 2:} } $A_2\nleq B$.

\smallskip

\noindent \textsc{Case 2a: }$T_1\cap A_2\nleq B$.

\smallskip

\noindent       Since $T_1\cap A_2\nleq B$, we may write $A= B\times \langle a\rangle$, with $a\in T_1\cap A_2$. In particular, $a^\alpha=a=a^{-1}$. Consider $\beta:=\alpha_{|B}$, the restriction of $\alpha$ to $B$. As $\alpha$ is not the identity or $\iota$ but fixes the involution $a$, $\beta$ is neither the identity nor the inverse automorphism of $B$. Now every subset $S\subseteq A\setminus B$ satisfying $S=S^{-1}$ and $S^\alpha=S$ is of the form $S=aT$, where $T$ is a subset of $B$ satisfying $T=T^{-1}$ and $T^\beta=T$. Observe that since $A=B\times \langle a\rangle$ and $A$ has exponent greater than $2$, the group $B$ has exponent greater than $2$. In particular we are in the position to apply Lemma 5.5  in \cite{DSV} to the group $B$. From \cite[Lemma 5.5]{DSV}, the number of choices for $T$ is at most $2^{\frac{11|B|}{24}+\frac{|B_2|}{2}}$. (Strictly speaking, \cite[Lemma 5.5]{DSV} only says that the number of choices for $T$ is at most $2^{\frac{11|B|}{24}+\frac{|B_2|}{2}+(\log|B|)^2}$, in our application here we may delete the extra factor $2^{(\log_2|B|)^2}$ because the automorphism $\beta$ of $B$ has been fixed.) Therefore
$$s\le 2^{\frac{11|B|}{24}+\frac{|B_2|}{2}}=2^{\frac{11|A|}{48}+\frac{|A_2\setminus B|}{2}}.$$

\smallskip

\noindent\textsc{Case 2b: }$T_1\cap A_2\leq B$.

\smallskip

\noindent Since $T_1\cap A_2=T_{-1}\cap A_2$ and $T_1\cap T_{-1}\le A_2$, the sets $T_1\setminus B$, $T_{-1}\setminus B$ and $A_2\setminus B$ are pairwise disjoint. For simplicity, we write $$\bar{T}_1:=T_1\setminus B,\,\,\, \bar{T}_{-1}:=T_{-1}\setminus B\,\,\,\textrm{ and }\,\,\,\bar{A}_2:=A_2\setminus B.$$ Observe that $\langle \alpha,\iota\rangle$ has orbits of size at least $2$ on $\bar{T}_1\cup\bar{T}_{-1}\cup \bar{A}_2$ and of size at least $4$ on $(A\setminus B)\setminus (\bar{T}_1\cup\bar{T}_{-1}\cup \bar{A}_2)$. Therefore,

\begin{align*}
s&\le 2^{\frac{|(A\setminus B)\setminus (\bar{T}_1\cup \bar{T}_{-1}\cup \bar{A}_2)|}{4}+\frac{|\bar{T}_1|}{2}+\frac{|\bar{T}_{-1}|}{2}+\frac{|\bar{A}_2|}{2}}\\
&=2^{\frac{|(A\setminus B)|}{4}+\frac{|\bar{T}_1|}{4}+ \frac{|\bar{T}_{-1}|}{4}+ \frac{|\bar{A}_2|}{4}}\\
&=2^{\frac{|A|}{8}+\frac{|\bar{T}_1|}{4}+ \frac{|\bar{T}_{-1}|}{4}+ \frac{|A_2\setminus B|}{4}}.
\end{align*}
For $i\in\{1,-1\}$, write $a_i:=|A:T_i|$ if $T_{i}\nleq B$ and $a_i:=\infty$ otherwise. Now, a simple computation shows that either
\begin{description}
\item[(a)] $$\left(\frac{1}{8}+\frac{1}{8a_1}+\frac{1}{8a_{-1}}\right)\le\frac{11}{48},\textrm{ or}$$
\item[(b)] $|A:T_1|=|A:T_{-1}|=2$, $T_1\nleq B$ and $T_{-1}\nleq B$.
\end{description}
In {\bf (a)}, we have
$$\frac{|A|}{8}+\frac{|\bar{T}_1|}{4}+\frac{|\bar{T}_{-1}|}{4}=|A|\left(\frac{1}{8}+\frac{1}{8a_1}+\frac{1}{8a_{-1}}\right)\le \frac{11|A|}{48}$$
and the proof of this lemma follows in this case.

Suppose {\bf (b)} holds, that is, $$|A:T_1|=|A:T_{-1}|=2,\, T_1\nleq B\, \textrm{ and }T_{-1}\nleq B.$$ As $T_1\cap A_2\le  B$, $T_1$, $T_{-1}$ and  $B$ are three distinct subgroups of $A$ having index $2$ and containing the index four subgroup $T_1\cap T_{-1}\le A_2$.

Assume $|A:A_2|=4$.  Moreover, since $T_1/A_2$ and $T_{-1}/A_2$ are two distinct subgroups of $A/A_2$ of order $2$, we deduce $A/A_2\cong C_2\times C_2$. Then $A_2=T_1\cap T_{-1}$ and $A\cong C_4\times C_4\times C_2^{\ell}$ for some $\ell\ge 0$.  Since $B$ contains $A_2$, we have $B\cong C_4\times C_2^{\ell+1}$. Therefore $(A,B)$ is one of the pairs in the statement of Theorem~\ref{thrm:aa}, contradicting one of the hypotheses of this lemma.

Assume that $|A:A_2|=2$. Therefore $A\cong C_4\times C_2^\ell$, for some $\ell\ge 1$. (If $\ell=0$, then $A\cong C_4$. However, $C_4$ has a unique subgroup of index $2$, forcing $T_1=T_{-1}=A_2$.) Now, $T_1$, $T_{-1}$ and $A_2$ are three distinct subgroups of $A$ having index $2$ and containing $T_1\cap T_{-1}=T_{1}\cap A_2=T_{-1}\cap A_2$. As $A/(T_1\cap T_{-1})\cong C_2\times C_2$, the subgroup equals one of $T_1$, $T_{-1}$, $A_2$, contradicting one of the previous paragraphs. Therefore, this case does not arise.
\end{proof}

\begin{lemma}\label{l:MYMYMY}Let $A$ be an abelian group and let $ B$ be a subgroup of $A$ having index $2$. The number of triples $(C,Z,S)$ with
\begin{itemize}
\item $A=C\times Z$,
\item $C$ a cyclic subgroup of $A$ of order $t\ge 4$,
\item $Z$ an elementary abelian $2$-subgroup of $A$ and
\item $S\subseteq A\setminus B$ such that $S=S'\times S''$, for some $S'\in \{C,\emptyset,\{1\},C\setminus\{1\}\}$ and for some $S''\subseteq Z$,
\end{itemize} is at most $2^{\frac{|A|}{8}+2\log_2|A|-1}$.
\end{lemma}
\begin{proof}
Clearly, we may assume that $A =\langle\lambda\rangle\times Z'$, for some elementary abelian
$2$-subgroup $Z'$ and some cyclic subgroup $\langle\lambda\rangle$ of order $t\ge 4$, otherwise we have no triple. If $t$ is odd, then this decomposition is
unique. If $t$ is even, then the number of choices for $C$ is $|Z'|$ because the subgroup $C$ equals $\langle\lambda k\rangle$, for some
$k\in Z'$; while the number of choices for $Z$ is at most the number of subgroups of
index $2$ in $\langle\lambda^{|\lambda|/2}\rangle\times Z'$, which is at most $2|Z'|$. Assume now that $C$ and $Z$ are fixed.

We have $4$ choices for $S'$. Moreover, if $S'=\emptyset$, then $S=S'\times S''=\emptyset$, for every subset $S''$ of $Z$. Therefore, when $S'=\emptyset$, we have only one choice for $S$. Similarly, when $S''=\emptyset$, we have only one choice for $S$. Therefore, let $S'\in\{C,\{1\},C\setminus\{1\}\}$ and let $S''\subseteq Z$ with $S''\ne \emptyset$ such that $S:=S'\times S''\subseteq A\setminus B$. As $S''\ne\emptyset$, $Z$ is not contained in $B$. Then $B\cap Z$ has index $2$ in $Z$ and hence we have $2^{|Z:Z\cap B|}=2^{\frac{|Z|}{2}}$ choices for $S''$. Since $S'$ equals $C$, $\{1\}$ or $C\setminus\{1\}$, we have $S'\cap B\ne\emptyset$ and hence there exists $s_1\in S'\cap B$. Since $s_1S''\subseteq A\setminus B$ and $s_1\in B$, we deduce $S''\subseteq Z\setminus B$.  Therefore, when $S'$ and $S''$ are both non-empty, we have at most $3\cdot 2^{\frac{|Z'|}{2}}$ choices for $S$.

It follows that there are most
$$|Z'|\cdot 2|Z'|\cdot (1+3\cdot 2^{\frac{|A|}{8}})\le 2^{2\log_2|Z'|+1+\frac{|A|}{8}+2}\le
2^{2\log_2(|A|/4)+1+\frac{|A|}{8}+2}\le2^{\frac{|A|}{8}+2\log_2|A|-1}$$
 triples.
\end{proof}

\begin{proof}[Proof of Theorem $\ref{thrm:11}$]
The group $A$ contains $2^{\frac{|A|}{4}+\frac{|A_2\setminus B|}{2}}$ inverse-closed subsets $S$ with $S\subseteq A\setminus B$ by Lemma~\ref{l:4}.

Now, we assume that Part~(2) does not hold, that is, $(A, B)$ is not one of the pairs in Theorem~\ref{thrm:aa}; we show that Part~(1) holds.
 If $A$ has exponent $2$, then $A_2=A$ and $|A_2\setminus B|=|A|/2$. Thus the result follows from Theorem \ref{thrm:1} because every Cayley digraph over an elementary abelian $2$-group is undirected. For the rest of the proof we assume that $A$ has exponent at least $3$, and hence the mapping $\iota:A\to A$ defined by $a^\iota=a^{-1}$, for every $a\in A$, is a non-identity automorphism of $A$.

We partition the set $\mathcal{S}:=\{S\mid S\subseteq A\setminus B,S=S^{-1}\}$ in five (not necessarily disjoint)  subsets:
\begin{align*}
\mathcal{A}_1:=&\{S\in \mathcal{S}\mid \langle S\rangle <A\},\\
\mathcal{A}_2:=&\{S\in \mathcal{S}\mid \textrm{there exists }\alpha\in \Aut(A) \textrm{ with }\alpha\ne 1, \alpha\ne\iota, S^\alpha=S\textrm{ and } B^\alpha= B\},\\
\mathcal{A}_3:=&\{S\in \mathcal{S}\mid \textrm{there exist two subgroups }H \textrm{ and }K \textrm{ with }1<H\le K<A, H\le B,\\
&\quad |H| \textrm{ and }|A:K| \textrm{ both prime numbers,} \textrm{ and }S\setminus K \textrm{ is a union of }H\textrm{-cosets}\}\\
&\textrm{ when }|A| \textrm{ is not a }2\textrm{-group, and }\\
\mathcal{A}_3:=&\emptyset \textrm { when }A \textrm{ is a }2\textrm{-group,}\\
\mathcal{A}_4:=&
\{
S\in\mathcal{S}
\mid
\textrm{there exist a cyclic subgroup }
C
 \textrm{ of order at least } 4 \textrm{ and an elementary }\\
&
  \textrm{abelian }2\textrm{-subgroup }Z\textrm{ with } A=C\times Z, \textrm{ there exist }\\
&
  S'\in \{
  \emptyset,
  \{1\}
  ,C,
  C\setminus
  \{1\}
  \} \textrm{ and }S''\subseteq Z \textrm{ with }S=S'\times S''
\},\\
\mathcal{A}_5:=&\mathcal{S}\setminus (\mathcal{A}_1\cup\mathcal{A}_2\cup \mathcal{A}_3\cup\mathcal{A}_4).
\end{align*}

From Lemma~\ref{l:3new},
\begin{equation}\label{eq:4new}|\mathcal{A}_1|\le 2^{\frac{|A|}{8}+\frac{|A_2\setminus B|}{2}+\log_2|A|}.
\end{equation} Observe that, if $S\in \mathcal{S}\setminus \mathcal{A}_1$, then $\Cay(A,S)$ is connected and hence $\{ B,A\setminus  B\}$ is the only bipartition of $\Cay(A,S)$. In particular, every automorphism of $\Cay(A,S)$ must preserve the bipartition  $\{ B, A\setminus B\}$.

From Lemma~\ref{l:22},
\begin{equation}\label{eq:5new}
|\mathcal{A}_2|\le 2^{\frac{11|A|}{48}+\frac{|A_2\setminus B|}{2}}(|\Aut(A)|-1)\le 2^{\frac{11|A|}{48}+\frac{|A_2\setminus B|}{2}+(\log_2|A|)^2}.
\end{equation}

Since $A$ contains at most $|A|^2$ subgroups $H$ and $K$ with $|H|$ and $|A:K|$ both prime numbers, Lemma~\ref{l:21} yields
\begin{equation}\label{eq:6new}
|\mathcal{A}_3|\le 2^{\frac{11|A|}{48}+\frac{|A_2\setminus B|}{2}+2\log_2|A|}.
\end{equation}

From Lemma~\ref{l:MYMYMY}, we have

\begin{equation}\label{eq:MYMYMY}
|\mathcal{A}_4|\le 2^{\frac{|A|}{8}+2\log_2|A|-1}.
\end{equation}

\smallskip

\noindent\textsc{Claim: } For every $S\in \mathcal{A}_5$, $\Cay(A,S)$ is a  bipartite Cayley graph on $A$ with $c(A,B)=2$.

\smallskip

\noindent Let $S\in \mathcal{A}_5$, let $\Gamma:=\Cay(A,S)$ and let $G:=\Aut(\Gamma)$. As $S\notin \mathcal{A}_1$, $\Gamma$ is connected, bipartite and $\{B,A\setminus B\}$ is the only bipartition of $\Gamma$.

Since $\Gamma$ is a Cayley graph over $A$, the group $A$ is embedded in $G$ via its right regular representation. Thus we may identify $A$ as a subgroup of $G$, and we do so. Let $G_1$ be the stabilizer of the vertex $1$ of $\Gamma$. Since  $1\in B$, the group $G_1$ fixes the two parts $B$ and $A\setminus B$ of the bipartition of $\Gamma$, that is, $B^\alpha=B$ for each $\alpha\in G_1$.

Let  $N:=\nor{G_1}A $. Clearly, $N=\{\alpha\in \Aut(A)\mid S^\alpha=S\}=\langle\iota\rangle$, because $S\notin\mathcal{A}_2$. Therefore, $\nor G A $ is a generalized dihedral group over the abelian group  $A$, that is, $\langle A,\iota\rangle=\nor G A$.  Therefore,  we are in the position to apply~\cite[Theorem~$4.3$]{DSV} to the group $G$, see also~\cite[Definition $4.1$]{DSV}.

Part (1) of Theorem $4.3$ in~\cite{DSV} does not hold for $G$ because $S\notin\mathcal{A}_3$ (observe here that arguing as in the proof of Theorem~\ref{thrm:1} we may assume that $|H|$ and $|A:K|$ are both prime numbers in the statement of~\cite[Theorem~$4.3$]{DSV}). Similarly, part (2) of Theorem $4.3$ in~\cite{DSV} also does not hold because $S\notin\mathcal{A}_4$. Therefore, from  Theorem~$4.3$ in~\cite{DSV}, we deduce $G=\nor G A=\langle A,\iota\rangle$. Thus $|\Aut(\Cay(A,S)):A|=2$ and so $c(A,B)=2$.~$_\blacksquare$

\smallskip

Now, the proof  follows immediately from the previous claim, together with  a computation using~\eqref{eq:4new},~\eqref{eq:5new},~\eqref{eq:6new} and~\eqref{eq:MYMYMY}.
\end{proof}

\begin{proof}[Proof of Corollary $\ref{cor:22}$]
Let $A$ be an abelian group and let $B$ be a subgroup of $A$ having index $2$. If $A$ has exponent $2$, then the proof follows from Corollary~\ref{cor:2}. Suppose that $A$ has exponent at least $3$.  If $|A|\ge 8\, 214$, then a computation shows that $|A|/4>11|A|/48+(\log_2|A|)^2+2$ and hence, by Theorem \ref{thrm:11}, there exists an inverse-closed subset $S\subseteq A\setminus B$ with $\Cay(A,S) $ having Cayley index $2$, that is, $c(A,B)=2$.

Suppose first $|A|\le 4\,096$. In this case the proof follows with the invaluable help of the computer algebra system \texttt{magma} \cite{magma}. All the computations are straightforward and use the same method explained in the proof of Corollary~\ref{cor:2}. Except for the pairs listed in Table \ref{table:2}, we generate at random $10\,000$ inverse-closed subsets $S$ of $A\setminus B$  and we check whether $\Cay(A,S)$ has Cayley index $2$: in all cases, we have shown that $c(A,B)=2$. When $(A,B)$ is one of the pairs in Table~\ref{table:2}, we construct all inverse-closed subsets $S$ of $A\setminus B$ and we compute $c(A,B)$.

Suppose then $|A|> 4\,096$ and $|A|<8\,214$.
Following the argument in the proof of Theorem~\ref{thrm:11}, Eqs.~\eqref{eq:4new},~\eqref{eq:5new},~\eqref{eq:6new} and~\eqref{eq:MYMYMY}, we see that there exists a subset $S\subseteq A\setminus B$ with $S=S^{-1}$ and with $\Cay(A,S) $ having Cayley index $2$ as long as
$$2^{\frac{|A|}{4}+\frac{|A_2\setminus B|}{2}}>
2^{\frac{|A|}{8}+\frac{|A_2\setminus B|}{2}+\log_2|A|}+
2^{\frac{11|A|}{48}+\frac{|A_2\setminus B|}{2}}|\Aut(A)|+
2^{\frac{11|A|}{48}+\frac{|A_2\setminus B|}{2}+2\log_2|A|}+
2^{\frac{|A|}{8}+2\log_2|A|-1}.$$
With \texttt{magma}, we have checked this inequality computing explicitly $|\Aut(A)|$; every  abelian group $A$  with $4\,096<|A|<8\,214$ satisfies this inequality.
\end{proof}

\begin{proof}[Proof of Theorem~$\ref{thrm:unlabelled2}$]
Let $c:=1$ when $A$ has exponent $2$ and let $c:=2$ when $A$ has exponent greater than $2$. Let $\iota:A\to A$ be the automorphism of $A$ with $a^\iota=a^{-1}$, for every $a\in A$.
For the proof, we let $\mathrm{GRR}(A,B)$ denote the set of unlabelled bipartite Cayley graphs over $A$ with bipartition $\{B,A\setminus B\}$ and having Cayley index $c$. Also, we let  $2^{A\setminus B}_{\mathrm{GRR}}$ be the collection of the subsets $S$ of $A\setminus B$ with $\Cay(A,S)$ having Cayley index $2$.

Let $S_1$ and $S_2$ be in $2^{A\setminus B}_{\mathrm{GRR}}$ and let $\Gamma_1 := \Cay(A, S_1)$ and $\Gamma_2 := \Cay(A, S_2 )$. Suppose that $\Gamma_1\cong\Gamma_2$ and let $\varphi$ be a graph isomorphism
from $\Gamma_1$ to $\Gamma_2$. Without loss of generality, we may assume that $1^\varphi=1$. Note that $\varphi$ induces a group automorphism from $\Aut(\Gamma_1)=\langle A,\iota\rangle$ to $\Aut(\Gamma_2)=\langle A,\iota\rangle$. Using the fact that the pair $(A,B)$ is not one of the exceptional pairs described in Theorem \ref{thrm:aa}, we deduce $A^\varphi=A$.  In particular, $\varphi\in \Aut(A)$ and $S_1$ and $S_2$ are conjugate via an element of $\Aut(A)$. This shows that $$|\mathrm{GRR}(A,B)|\geq \frac{|2^{A\setminus B}_{\mathrm{GRR}}|}{|\Aut(A)|}.$$ By Theorem~\ref{thrm:11}, we have $$|2^{A\setminus B}_{\mathrm{GRR}}|\geq 2^{\frac{|A|}{4}+\frac{|A_2\setminus B|}{2}}-2^{\frac{11|A|}{48}+\frac{|A_2\setminus B|}{2}+(\log_2|A|)^2+2}.$$ Since $|\Aut(A)|\leq  2^{(\log_2|A|)^2}$, it follows that
$$|\mathrm{GRR}(A,B)|\geq  2^{\frac{|A|}{4}+\frac{|A_2\setminus B|}{2}-(\log_2|A|)^2}-2^{\frac{11|A|}{48}+\frac{|A_2\setminus B|}{2}+2}
\geq  2^{\frac{|A|}{4}+\frac{|A_2\setminus B|}{2}}(2^{-(\log_2|A|)^2}-2^{-\frac{|A|}{48}+2}).$$

From Lemma \ref{l:4}, the number of inverse-closed subsets of $A$ contained in $A\setminus B$ is $2^{\frac{|A|}{4}+\frac{|A_2\setminus B|}{2}}$. Therefore
$$2^{\frac{|A|}{4}+\frac{|A_2\setminus B|}{2}}\ge |\mathrm{GRR}(A,B)|.$$
 This completes the proof.
\end{proof}

\end{document}